\documentclass[11pt]{amsart}
\usepackage{amsmath}
\usepackage{amsthm}
\usepackage{amssymb}
\usepackage{amsfonts,mathrsfs}
\usepackage{stmaryrd}
\usepackage{amsxtra}  
\usepackage{epsfig}
\usepackage{verbatim}
\usepackage[all]{xy}

\usepackage{mathtools}
\usepackage{tikz}

\theoremstyle{plain}
\newtheorem{theorem}[equation]{Theorem}
\newtheorem{proposition}[equation]{Proposition}
\newtheorem{lemma}[equation]{Lemma}
\newtheorem{corollary}[equation]{Corollary}
\newtheorem{definition}[equation]{Definition}
\theoremstyle{remark}
\newtheorem{remark}[equation]{Remark}
\numberwithin{equation}{section}

\def\norm#1{\Vert#1\Vert}

\newcommand{\ca}{{\mathcal A}}

\newcommand{\cg}{{\mathcal G}}

\newcommand{\ci}{{\mathcal I}}

\newcommand{\ck}{{\mathcal K}}

\newcommand{\co}{{\mathcal O}}

\newcommand{\cR}{{\mathcal R}}
\newcommand{\cs}{{\mathcal S}}

\newcommand{\B}{{\mathbb B}}
\newcommand{\C}{{\mathbb C}}

\newcommand{\h}{{\mathbb H}}

\newcommand{\Q}{{\mathbb Q}}
\newcommand{\R}{{\mathbb R}}

\newcommand{\Z}{{\mathbb Z}}

\begin{document}

\title[Bergman subspaces]{Bergman subspaces and subkernels: degenerate $L^p$ mapping and zeroes}
\author{L. D. Edholm  \& J. D. McNeal}
\subjclass[2010]{32W05}
\begin{abstract}
Regularity and irregularity of the Bergman projection on $L^p$ spaces is established on a natural family of bounded, pseudoconvex domains. The family is parameterized by a real variable $\gamma$. 
A surprising consequence of the analysis is that, whenever $\gamma$ is irrational, the Bergman projection is bounded only for $p=2$.
 \end{abstract}
\thanks{Research of the second author was partially supported by a National Science Foundation grant.}
\address{Department of Mathematics, \newline University of Michigan, Ann Arbor, Michigan, USA}
\email{edholm@umich.edu}
\address{Department of Mathematics, \newline The Ohio State University, Columbus, Ohio, USA}
\email{mcneal@math.ohio-state.edu}

\maketitle 


\section*{Introduction}\label{S:intro}

For $\gamma > 0$, define the domain 

\begin{equation}\label{D:genHartogs}
\h_{\gamma} := \left\{(z_1, z_2) \in \C^2: |z_1|^{\gamma} < |z_2| < 1 \right\},
\end{equation}
and call $\h_{\gamma}$ the {\it power-generalized Hartogs triangle of exponent $\gamma$}. The domain $\h_1$ is the classical Hartogs triangle.

The primary purpose of this paper is to show that the Bergman projection of $\h_\gamma$, $\mathbf{B}_{\gamma}$, is $L^p$ bounded for only a restricted range of those $p\in(1,\infty)$. Our goal is to directly relate the $L^p$ boundedness to the exponent $\gamma$, and explain how this restricted range is tied to the boundary singularity at $(0,0)$.

The results here extend \cite{EdhMcN16},
which dealt with the special case of $\gamma \in\Z^+$. For general $\gamma$, the $L^p$ boundedness of $\mathbf{B}_{\gamma}$ turns out to be fundamentally different depending on whether
$\gamma\in\Q$ or not. The fact that this arithmetical property of $\gamma$ effects mapping properties of $\mathbf{B}_{\gamma}$ was surprising, and motivated the writing of this paper.

When $\gamma\in\Q$, $\mathbf{B}_{\gamma}$ is $L^p$ bounded for a non-degenerate interval of $p$ about 2:

\begin{theorem}\label{T:MainLpRational}
Let $m,n \in \Z^+$ with $\gcd (m,n) = 1$.  The Bergman projection $\mathbf{B}_{m/n}$ is a bounded operator from $L^p(\mathbb{H}_{m/n})$ to $A^p(\mathbb{H}_{m/n})$ if and only if $p \in \left(\frac{2m+2n}{m+n+1}, \frac{2m+2n}{m+n-1}\right)$.
\end{theorem}
\noindent However when $\gamma\notin\Q$, the $L^p$ mapping of $\mathbf{B}_{\gamma}$ completely degenerates:

\begin{theorem}\label{T:MainLpIrrational}
Let $\gamma > 0$ be irrational.  The Bergman projection $\mathbf{B}_{\gamma}$ is a bounded operator from $L^p(\mathbb{H}_{\gamma})$ to $A^p(\mathbb{H}_{\gamma})$ if and only if $p = 2$.

\end{theorem}

A secondary purpose of this paper is to show the Bergman kernel of $\h_\gamma$ has zeroes for all $\gamma\geq 2$. This extends a theorem in \cite{Edh16} for the cases $\gamma\in \Z^+$, $\gamma\ge2$. It was also shown in \cite{Edh16} that the Bergman kernel of $\h_{1/k}$, $k\in\Z^+$, does not have zeroes, i.e. $\h_{1/k}$ is a Lu Qi-Keng domain in the terminology of \cite{Boas96}. The complete answer to the question of when the Bergman kernels associated to $\h_\gamma$ have zeroes is thus reduced to the case in which $1<\gamma <2$, but this remains unsolved. See Remark \ref{R:LuQiKeng}.

The Bergman kernel of $\h_\gamma$ is computed in Section \ref{S:Decomp} by summing an orthonormal basis for $A^2\left(\h_\gamma\right)$. When $\gamma=\frac mn\in\Q^+$, the summation occurs by grouping together monomials based on their exponent's distance to a certain critical line in the lattice $\Z^2$. The geometric representation of this we call the lattice point diagram of $\h_\gamma$ and is described in Section \ref{S:Decomp}. This leads to $m$ sub-Bergman kernels

\begin{equation}\label{E:bergman_decomp*}
\B_{m/n}(z,w) = K_0(z,w) \oplus K_1(z,w) \oplus \cdots \oplus K_{m-1}(z,w),
\end{equation}
where $\B_{m/n}$ is the full Bergman kernel of $\h_\gamma$, and to explicit formulas for each subkernel $K_j$. It follows from these formulas that $\B_{m/n}(z,w)$ is a rational function of 
$(z,w)\in \h_\gamma\times\h_\gamma$. It is intriguing that the denominators of the sub-Bergman kernels $K_j$ are identical for $0\leq j\leq m-1$. In any case, once the expressions for
$K_j$ are in hand, the $L^p$ boundedness range of their associated operators, $\ck_j$, are proved following the methods used in \cite{EdhMcN16}.  Theorem \ref{T:MainLpRational} follows by taking the smallest range amongst all $\ck_j$, $0\leq j\leq m-1$.

The method used in \cite{Edh16} to compute $\B_{k}(z,w)$, $k\in\Z^+$, was different: there the first author used Bell's transformation formula for the Bergman kernel under proper maps \cite{Bell82c} and fairly elaborate algebraic manipulations to compute the Bergman kernel of $\h_k$. The difficulty in executing these algebraic arguments for $\gamma\notin\Z^+$,  the existence of the kernel decomposition \eqref{E:bergman_decomp*}, and the power of the lattice point diagram generally recommended the method used in Section  \ref{S:Decomp}.

When $\gamma\notin\Q$, the Bergman kernel of $\h_\gamma$ is not a rational function, but we obtain an explicit enough formula to do further analysis, in particular to determine the existence of zeroes (Theorem \ref{T:LQKgamma}) and to prove Theorem \ref{T:MainLpIrrational}.

There are two additional points about methods that seem noteworthy. First, some basic facts from number theory are used at multiple points in the arguments. Congruence of integers and residue systems occur in the proof of Theorem \ref{T:subKernelCalculation}, Proposition \ref{P:boundednessSubprojections}, and Proposition \ref{L:NonLpBoundednessRational}, while Dirichlet's theorem on rational approximation of $\gamma\notin\Q$ is used in Section \ref{S:Irrational}. Though these facts are elementary, they also seem intrinsically connected our results. For example, the fact that Dirichlet's theorem gives a quadratic estimate (in the denominator) between $\gamma\notin\Q$ and $\frac mn$ is crucial for our proof of Theorem \ref{T:MainLpIrrational} in Section \ref{S:Irrational}. Second, care is taken in subsection \ref{SS:TypeA} to identify the size estimates of a general kernel on $\h_{m/n}$ needed to conclude $L^p$ boundedness of its operator. The exponent $A$ in that subsection is the essential parameter. Other natural kernels on $\h_{m/n}$, e.g. the Szeg\" o kernel as well as non-holomorphic kernels, can thus be analyzed via subsection \ref{SS:TypeA}.

There are antecedents to Theorem \ref{T:MainLpRational}, besides  \cite{EdhMcN16}.  Note that the domains $\h_{\gamma}$ are pseudoconvex, but the boundary of $\h_{\gamma}$, $b\h_{\gamma}$, is not smooth. The serious singularity is at $(0,0)$, near which $b\h_{\gamma}$ is not the graph of a continuous function; points of the form $\left(e^{ia}, e^{ib}\right), a,b\in\R$ are also non-smooth, but of a milder, polydisc-like type. Lanzani and Stein \cite{LanSte04} studied different classes of domains $\Omega\subset\C$, classified by severity of non-smoothness of the boundaries. Limited $L^p$ boundedness of $\mathbf{B}_{\Omega}$, analogous to Theorem \ref{T:MainLpRational}, is shown for certain classes. Krantz and Peloso \cite{KraPel} showed that the Bergman projection on non-smooth versions of the worm domain has limited $L^p$ boundedness. In \cite{ChaZey16}, Chakrabarti and Zeytuncu proved the result corresponding to Theorem \ref{T:MainLpRational} for $\h_1$, with a different proof than in \cite{EdhMcN16}. In \cite{Chen14}, Chen considers a different generalization of the Hartogs triangle than our $\h_{\gamma}$ and establishes limited $L^p$ boundedness in that situation. Perhaps the most significant overlap with our work is \cite{ZeyThesis} and \cite{Zey13}. Zeytuncu constructs particular non-smooth Hartogs domains, some exhibiting the limited range of $L^p$ boundedness of the type in Theorem \ref{T:MainLpRational} and others with the degeneracy of Theorem \ref{T:MainLpIrrational}. However the differences between our results and \cite{ZeyThesis, Zey13} are also significant. Zeytuncu's degenerate $L^p$ boundedness stems from his domains having exponential cusps at their boundary: see \cite{Zey13}, Theorem 1.2, for the essential, weighted one-variable result (which is lifted to $\C^2$ in the usual way to give an unweighted result). Our domains $\h_\gamma$, on the other hand, have only a polynomial like singularity at $(0,0)$.  And the degenerate $L^p$ boundedness in Theorem \ref{T:MainLpIrrational} comes from the fact that $\gamma$ is not rational, rather than exponential vanishing at the boundary. A result that encompasses both our Theorem \ref{T:MainLpIrrational} and Zeytuncu's Theorem 1.2 in \cite{Zey13} is lacking, but would be very interesting.

For many classes of pseudoconvex domains with smooth boundary it is known that the Bergman projection maps $L^p$ boundedly for all $p\in(1,\infty)$, see \cite{PhoSte77, McNeal89, McNeal94b, McNSte94, NagRosSteWai89, LanSte12} and their references for the principal results. But recently, restricted $L^p$ boundedness similar to Theorem \ref{T:MainLpRational} has also been shown on smoothly bounded worm domains, \cite{BarSah12}. Versions of these domains  were originally defined in \cite{DieFor77-1}. We also mention an earlier result of Barrett, \cite{Bar84}, of a smoothly bounded {\it non-pseudoconvex} domain whose Bergman projection has a restricted range of $L^p$ boundedness.

\bigskip


\section{Notation}\label{S:prelim}

If $\Omega\subset\C^n$ is a domain, let $\co(\Omega)$ denote the holomorphic functions on $\Omega$. The standard $L^2$ inner product is denoted

\begin{equation}\label{D:innerProduct}
\left\langle f,g\right\rangle =\int_\Omega f\cdot \bar g\, dV,
\end{equation}
where $dV$ denotes Lebesgue measure on $\C^n$. For $p>0$, 
$$L^p(\Omega) =\left\{f\, :\,\,\left(\int_\Omega |f|^p\, dV\right)^{\frac 1p}:= \|f\|_p <\infty\right\}$$  denotes the usual Lebesgue space of $p$-th power integrable functions. When $p=2$ we drop the subscript on the norm, i.e. $\|f\|^2= \langle f,f\rangle$. The Bergman spaces are denoted $A^p(\Omega)=\co(\Omega)\cap L^p(\Omega)$.

The Bergman projection $\mathbf{B}_\Omega :L^2(\Omega)\longrightarrow A^2(\Omega)$ is the orthogonal projection operator.
It is elementary that this operator is self-adjoint with respect to the inner product \eqref{D:innerProduct}. The Bergman kernel, denoted $\B_\Omega(z,w)$, satisfies

\begin{equation*}
\mathbf{B}_\Omega f(z)=\int_\Omega \B_\Omega(z,w) f(w)\, dV(w),\qquad f\in L^2(\Omega).
\end{equation*}
Given an orthonormal Hilbert space basis \{$\phi_{\alpha}\}_{\alpha \in \ca}$ for $A^2(\Omega)$, the Bergman kernel is given by the following formula, 

\begin{equation}\label{E:BergmanInfiniteSum}
\B_\Omega(z,w) = \sum_{\alpha \in \ca}\phi_{\alpha}(z)\overline{\phi_{\alpha}(w)}.
\end{equation}

Recall that $\Omega \subset \C^n$ is a Reinhardt domain if for every $z = (z_1, z_2, \cdots, z_n) \in \Omega$, it also holds that $(e^{i \theta_1} z_1, e^{i \theta_2} z_2, \cdots, e^{i \theta_n} z_n) \in \Omega$, where $\theta_1, \theta_2, \cdots, \theta_n$ are arbitrary real numbers.  See \cite{JarPflBook08} for a detailed treatment of analysis on these domains.  Given a Reinhardt domain $\Omega$, define the {\it Reinhardt shadow} of $\Omega$ to be the set 
\[
\omega := \{(r_1, r_2,\cdots, r_n) \in \R^n : r_j \ge 0, (r_1, r_2,\cdots, r_n) \in \Omega \}.
\]

The power-generalized Hartogs triangles \eqref{D:genHartogs} are clearly Reinhardt domains. For these domains, $\mathbf{B}_\gamma$ and $\B_\gamma(z,w)$ will denote
$\mathbf{B}_{\h_\gamma}$ and $\B_{\h_\gamma}(z,w)$ respectively. As usual, the operator $\mathbf{B}_{\gamma}$ is extended to supersets of $L^2\left(\h_\gamma\right)$ by setting

\begin{equation*}
\mathbf{B}_\gamma f(z)=\int_{\h_\gamma} \B_\gamma (z,w) f(w)\, dV(w),
\end{equation*}
whenever the integral is defined. We still refer to $\mathbf{B}_\gamma$ as the Bergman projection, even when acting on $L^p\left(\h_\gamma\right)$ for $p \in (1,2)$.

Two pieces of notational shorthand will also be used. If $D$ and $E$ are functions depending on several variables, $D\lesssim E$ will signify that there exists a 
constant $K>0$, independent of relevant variables, such that $D\leq K\cdot E$. The independence of which variables will be specified (or clear) in context.
Also, $D\approx E$ stands for $D\lesssim E\lesssim D$.
 If $x\in\R$, $\lfloor x\rfloor$ will denote the greatest integer $\leq x$.


\section{Decomposing the Bergman space}\label{S:Decomp}

\subsection{Allowable indices}

Let $\gamma$ be any positive real number.  Since $\h_{\gamma}$ is Reinhardt, every $f\in\co\left(\h_{\gamma}\right)$ has a unique Laurent expansion 
\[
 f(z) =  \sum_{\alpha \in \mathcal{A}} a_{\alpha}z^{\alpha},
\]
where $\ca$ is the set of multi-indices $\{\alpha = (\alpha_1,\alpha_2) \in \Z^2 : \alpha_1 \ge 0 \}$.  Since $z_2 \ne 0$ on $\h_{\gamma}$, $\alpha_2$ is allowed to be any negative integer.   Imposing square integrability, however, restricts the range of $\alpha_2$ allowed in the sum.

\begin{definition}
Say a multi-index $\alpha$ is $\gamma$-allowable if the monomial $z^{\alpha} \in A^2(\h_{\gamma})$.  Let $\ca^2_{\gamma}$ denote the set of $\gamma$-allowable multi-indices.
\end{definition}

It follows that $\left\{z^{\alpha}: \alpha \in \ca^2_{\gamma}\right\}$ is an orthogonal basis for $A^{2}(\h_{\gamma})$.  We now determine the set $\ca^2_{\gamma}$, and calculate the norms of these monomials.

\begin{lemma}\label{L:gammaAllowableMultiIndex} For any $\gamma\in\R^+$,

\begin{itemize}
\item[(i)] $\mathcal{A}^2_{\gamma} = \left\{(\alpha_{1},\alpha_{2}): \alpha_{1} \ge 0,\ \alpha_1 + \gamma(\alpha_{2} +1) > -1 \right\}.$ 
\smallskip
\item[(ii)] For $\alpha \in \mathcal{A}^2_{\gamma}$,

\begin{equation}\label{E:gammaNorm}
c_{\gamma,\alpha}^{2} :=  \left\|z^{\alpha}\right\|^{2} = \frac{\gamma\pi^2}{(\alpha_{1} + 1)^{2} + \gamma(\alpha_{1} +1)(\alpha_{2}+1)}.
\end{equation}
\end{itemize}

\begin{proof}  Let $H_{\gamma}$ be the Reinhardt shadow of $\h_{\gamma}$.  Using polar coordinates,

\begin{align}\label{E:gammaNorm2}
	\int_{\h_{\gamma}} |z_{1}^{\alpha_1}z_{2}^{\alpha_2}|^2\;dV
	&= 4\pi^2\int_{H_{\gamma}} r_1^{2\alpha_{1}+1}r_2^{2\alpha_{2}+1}\;dr \notag\\
	&= 4\pi^2 \int_0^1 \int_0^{r_2^{\frac 1\gamma}} r_1^{2\alpha_{1}+1}r_2^{2\alpha_{2}+1}\;dr_1\,dr_2 \notag\\
	&=\frac{2\pi^2}{\alpha_1+1} \int_0^1 r_2^{\frac 1\gamma(2\alpha_1+2) + 2\alpha_2+1}\, dr_2.
\end{align}

This integral converges if and only if $\frac 1\gamma(2\alpha_1+2) + 2\alpha_2+1> -1$. Notice that
$$\frac 1\gamma(2\alpha_1+2) + 2\alpha_2+1> -1\quad \Leftrightarrow\quad \alpha_1 + \gamma(\alpha_{2} +1) > -1,$$
so (i) holds.

Furthermore, when the integral \eqref{E:gammaNorm2} converges, it equals 

\begin{equation*}
\frac{\pi^2}{\alpha_1+1}\cdot \frac 1{\frac 1\gamma(\alpha_1+1) +(\alpha_2 +1)}= \frac{\gamma\pi^2}{(\alpha_{1} + 1)^{2} + \gamma(\alpha_{1} +1)(\alpha_{2}+1)}.
\end{equation*}
Thus, \eqref{E:gammaNorm} holds.

\end{proof}
\end{lemma}

A similar result holds for $\ca^p_\gamma$, the multi-indices $\alpha$ such that $z^\alpha\in A^p\left(\h_\gamma\right)$. The direct analog of \eqref{E:gammaNorm2} shows that $(\alpha_1,\alpha_2)\in \ca^p_\gamma$ if and only if $\alpha_1\geq 0$ and

\begin{align}\label{E:p_allowable}
\frac 1\gamma &(p\alpha_1+2) + p\alpha_2+1> -1\notag \\
\Leftrightarrow\quad \alpha_2 &>-\frac 1\gamma\alpha_1-\frac 2p-\frac 2{\gamma\cdot p}.
\end{align}

\begin{remark} In \cite{Zwonek99}, Lemma 5, Zwonek has characterized the monomials in $\ca^2(R)$ for more general Reinhardt domains $R$ than our domains $\h_\gamma$. The characterization involves cone considerations in the logarithmic image of $R$. A similar characterization of the indices $\ca^p(R)$ is given in \cite{Zwonek00}. It is easily checked that condition \eqref{E:p_allowable} coincides with Zwonek's for the domains $\h_\gamma$.

\end{remark}

When $\gamma = \frac mn \in \Q^+$, the strict inequality defining $(\alpha_1,\alpha_2)\in\ca^2_{\gamma}$ can be re-expressed as a non-strict inequality: 
\begin{align}\label{E:allowableMultiIndices}
\mathcal{A}^2_{m/n} &= \left\{(\alpha_{1},\alpha_{2}): \alpha_{1} \ge 0,\ \alpha_1 + \frac mn(\alpha_{2} +1) > -1 \right\} \notag\\
&= \left\{(\alpha_{1},\alpha_{2}): \alpha_{1} \ge 0,\ n\alpha_1 + m\alpha_{2} \ge -m -n +1 \right\}.
\end{align}
The simple step of passing to this closed condition on $\alpha_2$ is crucial for our subsequent work in the rational case. Notice this step is not possible if
$\gamma \notin \Q^+$.

It is convenient to interpret the multi-indices in $\mathcal{A}^2_{m/n}$ geometrically, as an explicitly closed subset of the lattice $\Z^2$ using the second representation in \eqref{E:allowableMultiIndices}. Thus, $z^{\alpha} \in A^2(\h_{m/n})$ 
if and only if $\alpha_1 \ge 0$, and $\alpha = (\alpha_1, \alpha_2)\in\Z^2$ lies on or above the line
\begin{equation}\label{E:RationalThreshold}
\alpha_2 = -\frac nm \alpha_1 + \frac{1-n-m}{m}.
\end{equation}
Call this subset of $\Z^2$ the {\it lattice point diagram} associated to $\mathcal{A}^2_{m/n}$.
Monomials corresponding to the fourth quadrant of the lattice point diagram, i.e. those lattice points where $\alpha_2 <0$, play an essential role in the analysis to follow.
The boundary lines described by \eqref{E:RationalThreshold}, corresponding to $\gamma = 1,2,3$, are illustrated below.

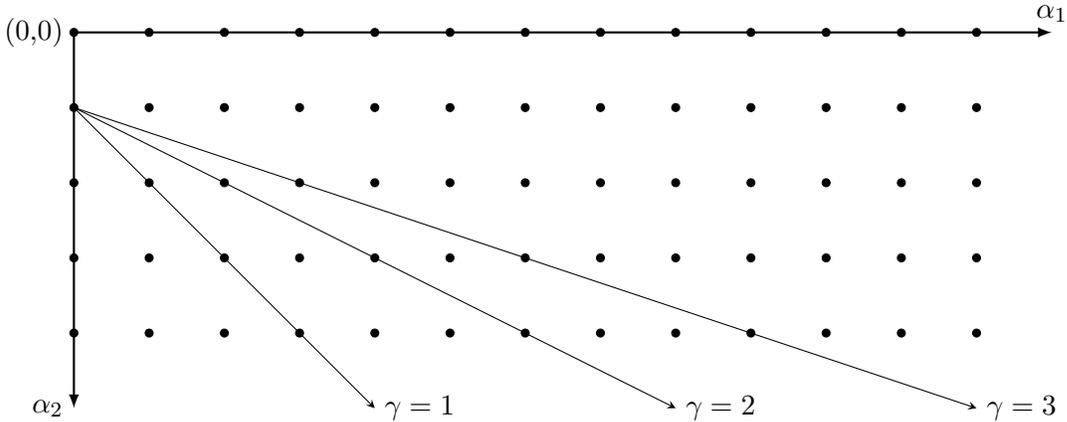
{\centering
\begin{tikzpicture}

\draw[-{latex}, thick] (0,0) -- (13,0) node[anchor=south] {$\alpha_1$};
\draw[-{latex}, thick] (0,0) -- (0,-5) node[anchor=east] {$\alpha_2$};

\draw[-stealth, thin] (0,-1) -- (4,-5) node[anchor=west] {$\gamma = 1$};
\draw[-stealth, thin] (0,-1) -- (8,-5) node[anchor=west] {$\gamma = 2$};
\draw[-stealth, thin] (0,-1) -- (12,-5) node[anchor=west] {$\gamma = 3$};

\filldraw[black] (0,0) circle (1.5pt) node[anchor=east] {(0,0)};
\filldraw[black] (1,0) circle (1.5pt) ;
\filldraw[black] (2,0) circle (1.5pt) ;
\filldraw[black] (3,0) circle (1.5pt) ;
\filldraw[black] (4,0) circle (1.5pt) ;
\filldraw[black] (5,0) circle (1.5pt) ;
\filldraw[black] (6,0) circle (1.5pt) ;
\filldraw[black] (7,0) circle (1.5pt) ;
\filldraw[black] (8,0) circle (1.5pt) ;
\filldraw[black] (9,0) circle (1.5pt) ;
\filldraw[black] (10,0) circle (1.5pt) ;
\filldraw[black] (11,0) circle (1.5pt) ;
\filldraw[black] (12,0) circle (1.5pt) ;

\filldraw[black] (0,-1) circle (1.5pt) ;
\filldraw[black] (1,-1) circle (1.5pt) ;
\filldraw[black] (2,-1) circle (1.5pt) ;
\filldraw[black] (3,-1) circle (1.5pt) ;
\filldraw[black] (4,-1) circle (1.5pt) ;
\filldraw[black] (5,-1) circle (1.5pt) ;
\filldraw[black] (6,-1) circle (1.5pt) ;
\filldraw[black] (7,-1) circle (1.5pt) ;
\filldraw[black] (8,-1) circle (1.5pt) ;
\filldraw[black] (9,-1) circle (1.5pt) ;
\filldraw[black] (10,-1) circle (1.5pt) ;
\filldraw[black] (11,-1) circle (1.5pt) ;
\filldraw[black] (12,-1) circle (1.5pt) ;

\filldraw[black] (0,-2) circle (1.5pt) ;
\filldraw[black] (1,-2) circle (1.5pt) ;
\filldraw[black] (2,-2) circle (1.5pt) ;
\filldraw[black] (3,-2) circle (1.5pt) ;
\filldraw[black] (4,-2) circle (1.5pt) ;
\filldraw[black] (5,-2) circle (1.5pt) ;
\filldraw[black] (6,-2) circle (1.5pt) ;
\filldraw[black] (7,-2) circle (1.5pt) ;
\filldraw[black] (8,-2) circle (1.5pt) ;
\filldraw[black] (9,-2) circle (1.5pt) ;
\filldraw[black] (10,-2) circle (1.5pt) ;
\filldraw[black] (11,-2) circle (1.5pt) ;
\filldraw[black] (12,-2) circle (1.5pt) ;

\filldraw[black] (0,-3) circle (1.5pt) ;
\filldraw[black] (1,-3) circle (1.5pt) ;
\filldraw[black] (2,-3) circle (1.5pt) ;
\filldraw[black] (3,-3) circle (1.5pt) ;
\filldraw[black] (4,-3) circle (1.5pt) ;
\filldraw[black] (5,-3) circle (1.5pt) ;
\filldraw[black] (6,-3) circle (1.5pt) ;
\filldraw[black] (7,-3) circle (1.5pt) ;
\filldraw[black] (8,-3) circle (1.5pt) ;
\filldraw[black] (9,-3) circle (1.5pt) ;
\filldraw[black] (10,-3) circle (1.5pt) ;
\filldraw[black] (11,-3) circle (1.5pt) ;
\filldraw[black] (12,-3) circle (1.5pt) ;

\filldraw[black] (0,-4) circle (1.5pt) ;
\filldraw[black] (1,-4) circle (1.5pt) ;
\filldraw[black] (2,-4) circle (1.5pt) ;
\filldraw[black] (3,-4) circle (1.5pt) ;
\filldraw[black] (4,-4) circle (1.5pt) ;
\filldraw[black] (5,-4) circle (1.5pt) ;
\filldraw[black] (6,-4) circle (1.5pt) ;
\filldraw[black] (7,-4) circle (1.5pt) ;
\filldraw[black] (8,-4) circle (1.5pt) ;
\filldraw[black] (9,-4) circle (1.5pt) ;
\filldraw[black] (10,-4) circle (1.5pt) ;
\filldraw[black] (11,-4) circle (1.5pt) ;
\filldraw[black] (12,-4) circle (1.5pt) ;
 
\end{tikzpicture}
}
\smallskip

\noindent The lattice point diagram indicates a useful way to decompose $A^2(\h_{m/n})$.  When $\gamma = \frac mn$, $\gcd(m,n)=1$, split the Bergman space into $m$ orthogonal subspaces
\begin{equation}\label{E:bergmanspace_decomp}
A^2(\h_{m/n}) = \cs_0 \oplus \cs_1 \oplus \cdots \oplus \cs_{m-1},
\end{equation}
where $\cs_j$ is the subspace spanned by monomials of the form $z^{\alpha}$, where $\alpha_1 \equiv j \mod m$.  Let 
\begin{equation}\label{E:Sjindices}
\cg_j=\left\{\alpha=(\alpha_1, \alpha_2)\in \ca^2_{m/n}: \alpha_1\equiv j\pmod m\right\}.
\end{equation}
That the decomposition \eqref{E:bergmanspace_decomp} is orthogonal follows from the fact that $\h_{m/n}$ is Reinhardt and $\cg_j\cap \cg_k=\varnothing$ if $j\neq k$.
Each $\cs_j$ is a closed subspace of $A^2(\h_{m/n})$, thus a Hilbert space.  Therefore the orthogonal projection, $L^2(\h_{m/n})\longrightarrow\cs_j$, is well-defined and represented by integration against a kernel, $K_j$.  It follows that
\begin{equation}\label{E:bergman_decomp}
.
\end{equation}
Call each $K_j$ a sub-Bergman kernel. In the next subsection, we shall focus on the subspaces $\cs_j$ and explicitly compute each $K_j$ in closed form.
For any rational exponent $\gamma$, \eqref{E:bergman_decomp} then implies an explicit expression for $\B_\gamma (z,w)$.

For irrational $\gamma$, the absence of a {\it finite} decomposition like \eqref{E:bergman_decomp} is the reason the methods in this paper do not imply an explicit closed form expression for 
the Bergman kernel of $\h_\gamma$. After $L^p$ mapping properties of the operators associated to the subkernels $K_j$ are proved, it will also be clear that the
lack of \eqref{E:bergman_decomp}  is the cause of the difference between Theorems \ref{T:MainLpRational} and \ref{T:MainLpIrrational}.

\subsection{Computing the sub-Bergman kernels}

Let $\gamma = \frac mn \in \Q^+$, $\gcd (m,n)=1$.  For each $j=0, \dots, m-1$, let $K_j$ be the sub-Bergman kernel of $\B_{m/n}$ given by \eqref{E:bergman_decomp} and $\cs_j$ the subspace in \eqref{E:bergmanspace_decomp}.  By definition, $\{z^{\alpha} c_{\gamma,\alpha}^{-1}: \alpha \in \cg_j \}$ is an orthonormal basis for $\cs_j$, where $\cg_j$ is given by \eqref{E:Sjindices} and $c_{\gamma,\alpha}$ by \eqref{E:gammaNorm}.  It follows that $K_j$ can be written as the following sum, which converges normally on $\h_{m/n} \times \h_{m/n}$:

\begin{equation}\label{E:subBergmanSum}
K_j(z,w) = \sum_{\alpha \in \cg_j} \frac{z^{\alpha}\bar{w}^\alpha}{c_{m/n,\alpha}^2}.
\end{equation}
We now compute this sum in closed form:

\begin{theorem}\label{T:subKernelCalculation}
Let $m,n \in \Z^+$ be relatively prime.  The sub-Bergman kernel $K_j$ of the domain $\h_{m/n}$ is given by 
\begin{align}\label{E:formulaK_j}
K_{j}(z,w) &= \frac{n}{m\pi^2} \cdot \frac{f_j(s,t)g_j(t) s^jt^{n-1- E_j}}{(1-t)^2(t^n-s^m)^2},
\end{align}
where $s = z_1\bar{w}_1, t = z_2\bar{w}_2, E_{j} = \left\lfloor \frac{(j+1)n-1}{m} \right\rfloor$, and $f_j$ and $g_j$ are the polynomials
\begin{align}
f_j(s,t) &= (j+1)t^n + (m - j-1)s^m, \label{E:NumPolynomialf} \\ 
g_j(t) &= \left(j+1 - \frac mn E_j\right) + \left(\frac mn + \frac mn E_j -j-1 \right)t.\label{E:NumPolynomialg}
\end{align}
\end{theorem}

\begin{proof} 

  First we find $K_j(z,z)$, then use polarization to move off the diagonal.  Working on the diagonal bypasses the ambiguity of raising a complex number to a fractional exponent.  Therefore, until the last two lines of the proof, let $s = |z_{1}|^{2}, t = |z_{2}|^{2}$. Also fix $t^{1/m}$ to be the positive real root.

Starting from \eqref{E:subBergmanSum} and using (\ref{E:gammaNorm}) and (\ref{E:RationalThreshold}),
\begin{align}
	K_j(z,z) 
	&= \sum_{\alpha \in \cg_j} \frac{s^{\alpha_1} t^{\alpha_2}}{c^2_{m/n,\alpha}} \notag\\
	&=\frac{n}{m\pi^{2}}\sum_{\alpha_1 \in \cR_j} \sum \ \left[(\alpha_{1} + 1)^{2} + \frac{m}{n}(\alpha_{1} +1)(\alpha_{2}+1)\right] s^{\alpha_{1}} t^{\alpha_{2}} \label{E:DoubleSum1},
\end{align}
where $\mathcal{R}_j := \{\alpha_{1} \ge 0: \alpha_{1} = j \mod m\}$ and the  inner sum is taken over integers $\alpha_2$ with $\alpha_2 \ge -\frac{n\alpha_{1}}{m} + \frac{1-n-m}{m}$.  We want to compute the smallest such integer, called $\ell(j)$.  Notice that

\begin{align*}
-\frac{n\alpha_{1}}{m} + \frac{1-n-m}{m} = -1 -\frac{n(\alpha_1-j)}{m} -\frac{(j+1)n-1}{m},
\end{align*}
and since $\alpha_{1}\equiv j \mod m$, it follows that

\begin{equation}\label{E:ell(j)}
\ell(j) = -1 - \frac{n(\alpha_{1}-j)}{m} - E_j.
\end{equation}
Therefore,

\begin{align*}
\eqref{E:DoubleSum1} 
	&= \frac{n}{m\pi^{2}}  \sum_{\alpha_{1} \in \mathcal{R}_j} \sum_{\alpha_{2} = \ell(j)}^{\infty} \ \left[(\alpha_{1} + 1)^{2} + \frac mn (\alpha_{1} +1)(\alpha_{2}+1)\right] s^{\alpha_{1}} t^{\alpha_{2}}\\
 	&= \frac{n}{m\pi^2} \sum_{\alpha_{1} \in \mathcal{R}_j} \sum_{\alpha_{2} = \ell(j)}^{\infty}(\alpha_{1} + 1)^{2}s^{\alpha_{1}} t^{\alpha_{2}} + \frac{1}{\pi^2} \sum_{\alpha_{1} \in \mathcal{R}_j} \sum_{\alpha_{2} = \ell(j)}^{\infty}(\alpha_{1} +1)(\alpha_{2}+1) s^{\alpha_{1}} t^{\alpha_{2}}\\
	&:= \frac{n}{m\pi^2} I(j) + \frac{1}{\pi^2}J(j).
\end{align*}

It remains to compute the sums $I(j)$ and $J(j)$.  Let $u := st^{-n/m}$, and note that both $0 < |t| < 1$ and $|u| < 1$.  Summation of $I(j)$ is straightforward:

\begin{align}
I(j) =  \sum_{\alpha_{1} \in \mathcal{R}_j}(\alpha_{1} + 1)^{2}s^{\alpha_{1}} \sum_{\alpha_{2} = \ell(j)}^{\infty}t^{\alpha_{2}}
	&=  \frac{1}{1-t}\sum_{\alpha_{1} \in \mathcal{R}_j}(\alpha_{1} + 1)^{2}s^{\alpha_{1}}t^{\ell(j)} \notag\\
	&=  \frac{t^{nj/m-1-E_j}}{1-t} \cdot \sum_{\alpha_{1} \in \mathcal{R}_j}(\alpha_{1} + 1)^{2}u^{\alpha_{1}} \notag\\
	&=  \frac{t^{nj/m-1-E_j}}{1-t} \cdot \frac{d}{du}\left(u \frac{d}{du}\left(\frac{u^{j+1}}{1-u^m}\right)\right). \label{E:SumIj}
\end{align}

Summation of $J(j)$ is slightly more involved. First, split the sum into two pieces:

\begin{align*}
J(j) &=  \sum_{\alpha_{1} \in \mathcal{R}_j}(\alpha_{1} +1)s^{\alpha_{1}} \sum_{\alpha_{2} = \ell(j)}^{\infty}(\alpha_{2}+1)  t^{\alpha_{2}}\\
	&=  \sum_{\alpha_{1} \in \mathcal{R}_j}(\alpha_{1} +1)s^{\alpha_{1}}\left[\frac{t^{\ell(j)+1}}{(1-t)^2}+\frac{(\ell(j)+1)t^{\ell(j)}}{1-t}\right]\\
	&= \frac{t}{(1-t)^2} \sum_{\alpha_1 \in \cR_j}(\alpha_1+1)s^{\alpha_1}t^{\ell(j)} + \frac{1}{1-t} \sum_{\alpha_1 \in \cR_j} (\alpha_1+1)(\ell(j)+1)s^{\alpha_1}t^{\ell(j)} \\
	&:= J_1(j) + J_2(j).
\end{align*}
For the first piece, it follows

\begin{align}
	J_1(j) &= \frac{t^{nj/m-E_j}}{(1-t)^2}  \sum_{\alpha_1 \in \cR_j} (\alpha_1+1)u^{\alpha_1} \notag \\
	&= \frac{t^{nj/m-E_j}}{(1-t)^2} \cdot \frac{d}{du}\left(\frac{u^{j+1}}{1-u^m}\right) \label{E:SumJ1j}.
\end{align}
For the second piece,

\begin{align}
	J_2(j) &= \frac{t^{nj/m-1-E_j}}{1-t}  \sum_{\alpha_1 \in \cR_j}(\alpha_1+1)\left(\ell(j)+1\right)u^{\alpha_1} \notag\\
	&= \frac{t^{nj/m-1-E_j}}{1-t}  \left[\left(\frac{nj}{m}-E_j\right) \sum_{\alpha_1 \in \cR_j}(\alpha_1+1)u^{\alpha_1} - \frac nm \sum_{\alpha_1 \in \cR_j}(\alpha_1+1)\alpha_1 u^{\alpha_1} \right] \notag \\
	&= \frac{t^{nj/m-1-E_j}}{1-t}  \left[\left(\frac{nj}{m}-E_j\right) \cdot \frac{d}{du}\left(\frac{u^{j+1}}{1-u^m}\right) - \frac nm \cdot u\frac{d^2}{du^2}\left(\frac{u^{j+1}}{1-u^m}\right)\right] \label{E:SumJ2j}.
\end{align}
Using Leibniz's rule, \eqref{E:SumIj} and \eqref{E:SumJ2j} can be combined more simply as
\begin{align*}
	I(j) + \frac mn J_2(j) &= \frac{t^{nj/m-1-E_j}}{1-t} \left(j+1- \frac mn E_j \right) \cdot \frac{d}{du}\left(\frac{u^{j+1}}{1-u^m}\right).
\end{align*}
Combining this with \eqref{E:SumJ1j}, we now have
\begin{align}
	K_j(z,z) &= \frac{n}{m\pi^2} \left[I(j) + \frac mn J_2(j) + \frac mn J_1(j) \right] \notag \\
		&= \frac{n}{m\pi^2} \cdot g_j(t) \cdot \frac{t^{nj/m-1-E_j}}{(1-t)^2}  \cdot \frac{d}{du}\left( \frac{u^{j+1}}{1-u^m}\right)\label{E:KjPrederivative},
\end{align}
where $g_j(t) := j+1-\frac mn E_j + (\frac mn+\frac mn E_j -j-1)t $.

Finally,
\begin{align}
	\eqref{E:KjPrederivative} &= \frac{n}{m\pi^2} \cdot g_j(t) \cdot \frac{t^{nj/m-1-E_j}}{(1-t)^2}  \cdot \frac{u^j}{(1-u^m)^2} \cdot \left(j+1+(m-j-1)u^m \right) \notag\\
		&= \frac{n}{m\pi^2} \cdot g_j(t) \cdot \frac{s^j t^{-1-E_j}}{(1-t)^2} \cdot \frac{t^{2n}}{(t^n-s^m)^2} \cdot \left(j+1+(m-j-1)u^m \right) \notag\\
		&= \frac{n}{m\pi^2} \cdot f_j(s,t) g_j(t) \cdot \frac{s^j t^{n-1-E_j}}{(1-t)^2(t^n-s^m)^2} \label{E:KjFormula},
\end{align}
where $f_j(s,t) := (j+1)t^n+(m-j-1)s^m$.  This establishes the desired formula for $K_j(z,z)$.  

Polarization now gives the formula for $K_j(z,w)$, substituting $s=z_1\bar{w}_1$ and $t=z_2\bar{w}_2$ into equation \eqref{E:KjFormula}. See section 1.1.5 of  \cite{DAngelo_SCVRHS} for an explanation of polarization in this context.
\end{proof}

The decomposition \eqref{E:bergman_decomp} now yields

\begin{corollary}
Let $m,n \in \Z^+$ be relatively prime.  The Bergman kernel of $\h_{m/n}$ is the explicit rational function
\begin{equation*}
\B_{m/n}(z,w) = \sum_{j=0}^{m-1} K_j(z,w),
\end{equation*}
where each $K_j(z,w)$ is calculated in Theorem \ref{T:subKernelCalculation}.
\end{corollary}

\subsection{The Lu Qi-Keng Problem}
In general, the Lu Qi-Keng problem is to determine which domains $\Omega \subset \C^n$ have vanishing Bergman kernel.  See \cite{Boas00} for background and more information.  In \cite{Edh16}, the problem is solved for the domains $\h_k$ and $\h_{1/k}$, $k\in\Z^+$:

\begin{theorem}\label{T:LQK_1/K}
Let $k \in \Z^+$.  The Bergman kernel $\B_{1/k}(z,w)\neq 0$ $\forall (z,w)\in \h_{1/k}\times\h_{1/k}$.
\end{theorem}

\begin{theorem}\label{T:LQK_K}
Let $k \ge 2$ be a positive integer.  The Bergman kernel $\B_k(z,w)$ has zeroes inside $\h_k \times \h_k$.
\end{theorem}

Using the explicit form of the orthonormal basis on $A^2\left(\h_\gamma\right)$, Theorem \ref{T:LQK_K} can be extended to non-integer exponents

\begin{theorem}\label{T:LQKgamma}
Let $\gamma\geq 2$.  The Bergman kernel $\B_{\gamma}(z,w)$ has zeroes inside $\h_{\gamma} \times \h_{\gamma}$.
\end{theorem}

\begin{proof}
First let $\gamma = \frac mn > 2$ be rational, $\gcd(m,n)=1$.  Write, as before, $s=z_1\bar w_1$ and $t = z_2\bar{w}_2$.

For $j = 1,\dots, m-1$, the positive exponent of $s$ in \eqref{E:formulaK_j} shows that  
$$K_j((0,z_2),(0,w_2)) = 0.$$
Thus, all but the sub-Bergman kernel $K_0$ in the decomposition \eqref{E:bergman_decomp} vanish identically on the variety $\{s=0\}$.  For this sub-Bergman kernel,
\begin{equation*}
K_0((0,z_2),(0,w_2)) = \frac{n}{m\pi^2}\cdot \frac{1 + (\frac{m}{n}-1)t}{t(1-t)^2}.
\end{equation*}
The numerator obviously vanishes when $t = -(\frac{m}{n} -1)^{-1}$.  It is easily checked that  $$(z^0,w^0) = \left((0,i(\frac{m}{n} -1)^{-1/2}),(0,-i(\frac{m}{n} -1)^{-1/2})\right) \in \h_{m/n} \times \h_{m/n},$$ 
and that $\B_{m/n}(z^0,w^0) = 0$, so this case is complete.  

Now let $\gamma > 2$ be irrational.  $A^2(\h_{\gamma})$ does not admit the finite decomposition \eqref{E:bergman_decomp}, however a similar simplification to that used above occurs when $\B_{\gamma}(z,w)$ is restricted to the
variety $\{s=0\}$. Starting with

\begin{equation*}
\B_{\gamma}(z,w) = \frac{1}{\gamma\pi^2} \sum_{\alpha \in \ca_{\gamma}^2}[(\alpha_1+1)^2 + \gamma (\alpha_1+1)(\alpha_2+1)]s^{\alpha_1}t^{\alpha_2},
\end{equation*}
it follows that
\begin{align*}
\B_{\gamma}((0,z_2),(0,w_2)) &= \frac{1}{\gamma\pi^2} \sum_{\alpha_2=-1}^{\infty} t^{\alpha_2} + \frac{1}{\pi^2} \sum_{\alpha_2=-1}^{\infty} (\alpha_2+1)t^{\alpha_2} \\
	&= \frac{1}{\gamma\pi^2} \cdot \frac{1+(\gamma-1)t}{t(1-t)^2}.  
\end{align*}
Recalling that $\gamma > 2$, it is checked as before that $$\left(z^0,w^0\right) = \left((0,i(\gamma -1)^{-1/2}),(0,-i(\gamma -1)^{-1/2})\right) \in \h_{\gamma} \times \h_{\gamma}.$$ Since $\B_{\gamma}\left(z^0,w^0\right) = 0$ by inspection, this case is complete as well.

The case $\gamma=2$ is covered by Theorem \ref{T:LQK_K}. However the point obtained above for $\gamma >2$ does not work for $\B_2$, since $\left((0,i(\gamma -1)^{-1/2}),(0,-i(\gamma -1)^{-1/2})\right)$ lies on the boundary of $\h_{\gamma} \times \h_{\gamma}$ when $\gamma = 2$. But it is easy to check that $\B_2\left((\frac{i}{\sqrt{2}},\frac{\sqrt{7}+i}{4}),(\frac{-i}{\sqrt{2}},\frac{\sqrt{7}-i}{4})\right) = 0$ and this point lies inside $\h_2 \times \h_2$.

\end{proof}

\begin{remark}\label{R:LuQiKeng} In order to answer the Lu Qi-Keng question for all $\h_{\gamma}$, $\gamma >0$, we use the fact that a vanishing Bergman kernel is a biholomorphic invariant.  The map $\Psi(z) = (z_1z_2,z_2)$ is a biholomorphism of $\h_{\gamma}$ onto $\h_{\gamma/(\gamma+1)}$ with inverse $\psi(z_1,z_2)=(\frac{z_1}{z_2},z_2)$.  Applying $\Psi$ recursively yields the following chain of equivalent domains:

\begin{equation*}
\begin{matrix}
&\Psi&&\Psi&&\Psi \qquad \Psi&&\Psi&\\
\h_{\gamma}&\rightleftharpoons &\h_{\gamma/(1+\gamma)}&\rightleftharpoons&\h_{\gamma/(1+2\gamma)}&\rightleftharpoons \cdots \rightleftharpoons &\h_{\gamma/(1+k\gamma)} &\rightleftharpoons &\cdots \\
&\psi&&\psi&&\psi \qquad \psi&& \psi &
\end{matrix}
\end{equation*}

A similar chain of domains appeared in sections 3.2 and 4.1 of \cite{Edh16}.  Theorem \ref{T:LQKgamma} now implies the Bergman kernel of $\h_{\gamma}$ has zeroes for $ \gamma \in [\frac 23,1) \cup [\frac 25,\frac 12) \cup \cdots \cup [\frac{2}{2k+1},\frac{1}{k}) \cup \cdots$, $k \in \Z^+$.  The right end points appearing in this union are all sharp, since Theorem \ref{T:LQK_1/K} says the Bergman kernel is non-vanishing for $\gamma=\frac 1k$, $k\in\Z^+$.  

The remaining open case is for $\gamma\in (1,2) \cup (\frac 12,\frac 23) \cup \cdots \cup (\frac{1}{k},\frac{2}{2k-1}) \cup \cdots$, $k \in \Z^+$.  By considering the same chain of biholomorphisms above, it is sufficient to investigate the question for $\gamma\in (1,2)$.
\end{remark}


\section{The rational case: $L^p$ boundedness}\label{S:PositiveLpRational}

\subsection{Type-$A$ operators on $\h_{m/n}$}\label{SS:TypeA}

Our proof of $L^p$ boundedness of $\mathbf{B}_{m/n}$ does not use holomorphicity of the Bergman kernel. It only involves size estimates of $\B_{m/n}(z,w)$ so it also applies to a general class of
operators whose kernels satisfy these estimates. It turns out that the exponent, $A$, of the euclidean distance from $z\in\h_{m/n}$ to the origin in these estimates determines the range of $p$ for $L^p$ boundedness.  This motivates the definition below.

 If $\Omega \subset \C^n$ is a domain and $K$ is an a.e. positive, measurable function on $\Omega \times \Omega$, let $\ck$ denote the integral operator with kernel $K$:

\begin{equation}\label{E:op_kernel}
\ck(f)(z) = \int_{\Omega} K(z,w) f(w) \,dV(w).
\end{equation}

\begin{definition} For $A\in\R^+$, call $\ck$ an operator of type-$A$ on $\h_{m/n}$ if its kernel satisfies
\begin{equation}\label{E:typeA}
\left| K\left(z_1,z_2, w_1,w_2\right)\right|\lesssim \frac{\left| z_2\bar w_2\right|^A}{\left|1-z_2\bar w_2\right|^2\, \left| z_2^n\bar w_2^n-z_1^m\bar w_1^m\right|^2},
\end{equation}
for a constant independent of $(z,w)\in \h_{m/n}\times\h_{m/n}$.
\end{definition}

The basic $L^p$ mapping result is the following

\begin{proposition}\label{P:typeA_mapping}
If $\ck$ is an operator of type-$A$ on $\h_{m/n}$, then $\ck:L^p\left(\h_{m/n}\right)\longrightarrow L^p\left(\h_{m/n}\right)$ boundedly if
\begin{equation}\label{E:typeA_range}
\frac{2n+2m}{Am+2n+2m-2nm} < p <\frac{2n+2m}{2nm -Am},
\end{equation}
whenever both denominators in \eqref{E:typeA_range} are positive and ${Am+2n+2m-2nm} > {2nm -Am}$.
\end{proposition}

\begin{remark} (i) While seemingly complicated at first glance, the bounding terms on $p$ in \eqref{E:typeA_range} express the natural interplay between the exponent $A$ and the kind of singularity $b\h_{m/n}$ has at $(0,0)$. 
\smallskip

(ii) The exponents $A$ obtained for the sub-Bergman kernels will automatically satisfy the positivity conditions mentioned after \eqref{E:typeA_range}.
\smallskip

(iii) If $A\to 2n$, the bounding terms in  \eqref{E:typeA_range} tend to 1 and $\infty$ respectively. Thus, an operator of type-$2n$ on  $\h_{m/n}$ is $L^p$ bounded for all $1<p<\infty$. This holds for any $m\in\Z^+$.

(iv) The bounding terms in  \eqref{E:typeA_range} are conjugate H\" older exponents. If, in \eqref{E:typeA}, $|z_2\bar w_2|^A$ is replaced by $|z_2|^{c}|w_2|^d$ for $c\neq d$, this H\" older symmetry will be broken but a result similar to Proposition \ref{P:typeA_mapping} can be obtained
\end{remark}

Some preliminary results are needed before proving Proposition \ref{P:typeA_mapping}.

\subsubsection{An estimate on $\B_D$} Proving Proposition \ref{P:typeA_mapping} requires analyzing integrals over the domain $\h_{m/n}$. Since $\h_{m/n}$ is rotationally symmetric, 
a one-dimensional estimate on the Bergman kernel on the unit disc in $\C$ can be used to effectively estimate these two-dimensional integrals. 

The essential estimate below (without $|w|^{-\beta}$ in the integrand) has been re-discovered many times, see for instance \cite{ForRud74}, \cite{RudinFunctiontheory}, \cite{ZhuBergmanbook}, \cite{Chen14}.
The proofs in these sources use well-known, but non-trivial, asymptotic results to derive the estimate. A more elementary proof is presented here; this proof simplifies one given in \cite{EdhMcN16}.

\smallskip

\begin{proposition} \label{L:calculus1}
Let $D\subset\C$ be the unit disc, $\epsilon \in (0,1)$ and $\beta \in (-\infty,2)$.  

Then for $z\in D$,
$$\ci_{\epsilon,\beta}(z) := \int_D \frac{(1 - |w|^2)^{-\epsilon}}{|1 - z\bar{w}|^2}|w|^{-\beta} \, dV(w) \lesssim (1 - |z|^2)^{-\epsilon},$$
with constant independent of $z$.
\end{proposition}

\begin{proof}  Since $|w|^{-\beta}\leq 1$ if $\beta\in (-\infty, 0)$, this range of $\beta$ reduces to establishing the estimate for $\beta =0$. From now on, $\beta \in [0,2)$.

Consider first an arbitrary $|z| \le \frac{1}{2}$.  Then $|1 - z\bar{w}| \ge 1 - |z\bar{w}| \ge \frac{1}{2}$, so 

\begin{align*}
\ci_{\epsilon,\beta}(z) &\le 4 \int_{D} (1 - |w|^2)^{-\epsilon}|w|^{-\beta} \, dV(w) \\
	&= 4\pi \int_0^1 (1-u)^{-\epsilon} u^{-\beta/2} \, du < \infty.
\end{align*}
Since this bound is independent of $z$, the desired estimate holds.

Next consider $|z| > \frac{1}{2}$. Set $c = \frac{1}{2|z|}$ and split the integral:
\begin{align*}
\ci_{\epsilon,\beta}(z) &= \int_{|w| \le c} \frac{(1 - |w|^2)^{-\epsilon}}{|1 - z\bar{w}|^2}|w|^{-\beta} \, dV(w) + \int_{|w| > c} \frac{(1 - |w|^2)^{-\epsilon}}{|1 - z\bar{w}|^2}|w|^{-\beta} \, dV(w) \\
&:= I_1 + I_2.
\end{align*}
For $I_1$,  $|1 - z\bar{w}| \ge 1 - |z\bar{w}| \ge \frac{1}{2}$. Hence

\begin{align*}
I_1 \le 4 \int_{|w| \le c} (1 - |w|^2)^{-\epsilon}|w|^{-\beta} \, dV(w) 
	< 4 \int_{D} (1 - |w|^2)^{-\epsilon}|w|^{-\beta} \, dV(w) < \infty.
\end{align*}
Thus $I_1$ satisfies the required estimate.  

It remains to show that $I_2$ does too.  Since $\frac{1}{2} < |z| < 1$, obviously $\frac{1}{2} < c < 1$ and consequently  $\frac{1}{2} < |w| < 1$ throughout $I_2$.  For $\beta \in [0,2)$, it follows that $1 \le |w|^{-\beta} < 4$.  Thus
\[
I_2 \le 4 \int_{|w| > c} \frac{(1 - |w|^2)^{-\epsilon}}{|1 - z\bar{w}|^2} \, dV(w).
\]

Now
\begin{align}\label{E:adhoc1}
\int_{|w| > c} \frac{(1 - |w|^2)^{-\epsilon}}{|1 - z\bar{w}|^2} \, dV(w) 
	&= \int_c^1 r(1-r^2)^{-\epsilon} \left[ \int_0^{2\pi} \frac{d\theta}{1 - 2r|z|\cos{\theta}+ r^2|z|^2} \right] dr.
\end{align}
Evaluation of the integral in brackets may be done by residue calculus.  Let $a = 1+r^2|z|^2$, $b= 2r|z|$, $w=e^{i\theta}$ and $\Gamma$ denote the unit circle.
\begin{align*}
\int_0^{2\pi} \frac{d\theta}{1 - 2r|z|\cos{\theta}+ r^2|z|^2}
	&= \int_0^{2\pi} \frac{d\theta}{a - b\cos{\theta}}\\
	&= \frac{1}{i}\int_{\Gamma} \frac{w^{-1} dw}{a - b(w + w^{-1})/2}\\
	&= -\frac{2}{i}\int_{\Gamma} \frac{dw}{bw^2 -2aw + b}.
\end{align*}

The polynomial in the denominator has two roots, only one of which is contained in the unit circle.  Indeed, $a > b$ and
\begin{equation*}
\left| \frac{a- \sqrt{a^2-b^2}}{b} \right| = \left| \frac{b}{a+\sqrt{a^2-b^2}} \right| < 1.
\end{equation*}
Denote this root by $\zeta$.  Using the residue theorem and L'Hospital's rule,
\begin{align*}
-\frac{2}{i}\int_{\Gamma} \frac{dw}{bw^2 -2a + b} &= -4\pi \lim_{w \to \zeta} (w - \zeta)\cdot \frac{1}{bw^2- 2aw + b}\\
&= -\frac{2\pi}{b\zeta-a}\\
&= \frac{2\pi}{\sqrt{a^2-b^2}} = \frac{2\pi}{(1-r^2|z|^2)}.
\end{align*}

Therefore, returning to \eqref{E:adhoc1},
\begin{equation*}
I_2 \lesssim \int_c^1 r(1-r^2)^{-\epsilon} (1 - r|z|)^{-1} \, dr.
\end{equation*}
A trivial over-estimate of $I_2$ now yields the desired estimate:
\begin{align*}
I_2 &\lesssim \int_0^{|z|} r(1-r^2)^{-\epsilon} (1 - r|z|)^{-1} \, dr + \int_{|z|}^1 r(1-r^2)^{-\epsilon} (1 - r|z|)^{-1} \, dr \\
	&:= J_1 + J_2
\end{align*}
Since $0 \le r \le |z|$, it follows that
\begin{align*}
J_1 < \int_0^{|z|} r(1-r|z|)^{-\epsilon-1} \, dr &= -\frac{1}{|z|} \int_1^{1-|z|^2} u^{-\epsilon-1} \, du \\
	&\lesssim (1 - |z|^2)^{-\epsilon}.
\end{align*}
The fact that $r \le 1$ implies
\begin{align*}
J_2 < (1 - |z|)^{-1} \int_{|z|}^1 r(1 - r^2)^{-\epsilon} \, dr &\lesssim (1 - |z|)^{-1} (1 - |z|^2)^{1 - \epsilon} \\
	&\lesssim (1 - |z|^2)^{-\epsilon}.
\end{align*}
Together, these estimates show $I_2 \lesssim (1 - |z|^2)^{-\epsilon}$, which completes the proof.
\end{proof} 

\subsubsection{An extension of Schur's lemma} The sub-Bergman kernels $K_j$ are not uniformly in $L^1\left(\h_{m/n}\right)$, i.e., there is no constant {\it independent of $z$} such that
$$\int_{\h_{m/n}} \left| K_j(z,w)\right|\, dV(w) \leq C.$$
See \eqref{E:SubkernelEstimate} below.  This prevents a direct application of H\" older's inequality from implying $L^p$ boundedness of $\ck_j$. 

A variant of Schur's lemma, proved in \cite{EdhMcN16}, will instead be used to prove $L^p$ boundedness. The difference between this
result and Schur's classical lemma (see, e.g., \cite{McNSte94}) is the explicit relationship between the range of exponents of the test function $h$ and the range of $p$ for which $L^p$ boundedness can be concluded.

\begin{lemma}[Version of Schur's Lemma \cite{EdhMcN16}]\label{L:SchursLemma}  Let $\Omega \subset \C^n$, $K$ and $\ck$ associated via \eqref{E:op_kernel}. 

Suppose there exists a positive auxiliary function $h$ on $\Omega$, and numbers $0 < a < b$ such that for all $\epsilon \in [a,b)$, the following estimates hold:

\begin{itemize}
\item[(i)]$\ck(h^{-\epsilon})(z)  \lesssim h(z)^{-\epsilon}$  and
\item[(ii)]$\ck(h^{-\epsilon})(w)  \lesssim h(w)^{-\epsilon},$
\end{itemize}
with constants independent of $z, w\in\Omega$.

Then $\ck$ is a bounded operator on $L^p(\Omega)$, for all $p \in (\frac{a+b}{b},\frac{a+b}{a})$.
\end{lemma}

As with other versions of Schur's lemma, the inherent advantage of Lemma \ref{L:SchursLemma} is the latitude of choosing the auxiliary function $h$. 

\subsubsection{The auxiliary function; proof of Proposition \ref {P:typeA_mapping}} On the power-generalized Hartogs triangle $\h_{m/n}$, define 

\begin{equation}\label{E:aux_function}
h(z) = \left(|z_2|^{2n} - |z_1|^{2m}\right)(1 - |z_2|^2).
\end{equation}
This function (essentially) measures the distance of $z \in \h_{m/n}$ to $b \h_{m/n}$.

\begin{proof}[Proof of Proposition \ref {P:typeA_mapping}] Let $\ck$ be an operator of type-$A$ on $\h_{m/n}$; assume \newline $Am+2n+2m-2nm, 2nm -Am>0 $ and ${Am+2n+2m-2nm} > {2nm -Am}$.

Let $\epsilon >0$ be momentarily unrestricted; restrictions on $\epsilon$ will emerge shortly. From \eqref{E:typeA}

\begin{align}\label{E:adhoc2}
\ck\left(h^{-\epsilon}\right)(z) &\lesssim \int_{\h_{m/n}} \frac{|z_{2}\bar{w}_{2}|^A (|w_2|^{2n} - |w_1|^{2m})^{-\epsilon}(1 - |w_2|^2)^{-\epsilon}}{|1-z_{2}\bar{w}_{2}|^2|z_{2}^n\bar{w}_{2}^n-z_{1}^m\bar{w}_{1}^m|^2}\,dV(w)\notag \\
	&= \int_{D^*} \frac{|z_{2}\bar{w}_{2}|^A (1 - |w_2|^2)^{-\epsilon}}{|1-z_{2}\bar{w}_{2}|^2} \left[ \int_{W} \frac{(|w_2|^{2n} - |w_1|^{2m})^{-\epsilon}}{|z_{2}^n\bar{w}_{2}^n-z_{1}^m\bar{w}_{1}^m|^2} \, dV(w_1)\right] dV(w_2).
\end{align}
Here $D^* = \{w_2:0 < |w_2| < 1\}$ and the region $W = \{w_1 : |w_1| < |w_2|^{n/m}\}$, where $w_2$ is considered fixed.  Denote the integral in brackets by $I$. Then 

\begin{align*}
I = \frac{1}{|z_2|^{2n} |w_2|^{2n + 2n\epsilon}} \int_{W} \left(1 - \left|\frac{w_1^m}{w_2^n}\right|^2 \right)^{-\epsilon} \left| 1 - \frac{z_1^m\bar{w}_1^m}{z_2^n\bar{w}_2^n}\right|^{-2}  dV(w_1).
\end{align*}

Make the substitution $u = \frac{w_1^m}{w_2^n}$.  This transformation sends $W$ to $m$ copies of $D$, the unit disc in the $u$-plane.  Proposition \ref{L:calculus1} yields

\begin{align*}
I &= \frac{|w_2|^{2n/m - 2n - 2n\epsilon}}{m |z_2|^{2n}} \int_{D} \frac{(1 - |u|^2)^{-\epsilon}}{\left|1 - z_1^m z_2^{-n} \bar{u}\right|^2} \cdot |u|^{2/m -2} \, dV(u) \\
	&\lesssim \frac{|w_2|^{2n/m -2n -2n\epsilon}}{|z_2|^{2n}} \left(1 - \left| \frac{z_1^m}{z_2^n}\right|^2 \right)^{-\epsilon} \\
	&=  \frac{|w_2|^{2n/m -2n -2n\epsilon}}{|z_2|^{2n-2n\epsilon}} \left(|z_2|^{2n} - |z_1|^{2m}\right)^{-\epsilon}.
\end{align*}

Returning to \eqref{E:adhoc2}, we have

\begin{equation*}
\ck\left(h^{-\epsilon}\right)(z) \lesssim  |z_2|^{A +2n\epsilon - 2n} \left(|z_2|^{2n} - |z_1|^{2m}\right)^{-\epsilon} \int_{D^*} \frac{(1 - |w_2|^2)^{-\epsilon}}{|1-z_{2}\bar{w}_{2}|^2} |w_2|^\beta \, dV(w_2),
\end{equation*}
where $\beta= A+2n/m - 2n - 2n\epsilon$. This will be favorably estimated by Proposition \ref{L:calculus1} if $\beta >-2$. That is, if

\begin{equation}\label{E:epsilon1}
\epsilon < \frac 1{2n}\left[ A+\frac{2n}m -2n +2\right],
\end{equation}
then

\begin{align*}
\ck\left(h^{-\epsilon}\right)(z) &\lesssim  |z_2|^{A +2n\epsilon - 2n} \left(|z_2|^{2n} - |z_1|^{2m}\right)^{-\epsilon} (1 - |z_2|^2)^{-\epsilon} \\
&= |z_2|^{A +2n\epsilon - 2n}\cdot h(z)^{-\epsilon}.
\end{align*}
In order for the first factor in this expression to be bounded, the exponent must be non-negative, i.e.

\begin{equation}\label{E:epsilon2}
\epsilon\geq 1-\frac A{2n}.
\end{equation}

Thus,  if $ a=1-\frac A{2n}$ and $b=\frac 1{2n}\left[ A+\frac{2n}m -2n +2\right]$, the above shows that 

$$\ck\left(h^{-\epsilon}\right)(z) \lesssim h(z)^{-\epsilon}\qquad\forall\epsilon\in [a,b).$$
Lemma \ref{L:SchursLemma}, and elementary algebra on the endpoints $a, b$, then show that $\ck$ is bounded on $L^p$ for the range of $p$ stated in Proposition \ref{P:typeA_mapping}.
\end{proof}

\subsection{Mapping properties of sub-Bergman projections}

From the polynomial expressions (\ref{E:NumPolynomialf}), (\ref{E:NumPolynomialg}) and the fact that $|s|^m < |t|^n <1$ when $(z,w)\in\h_{m/n}\times \h_{m/n}$, the estimates

\begin{align*}
|f_j(s,t)| \lesssim |t|^n, \qquad |g_j(t)| \lesssim 1,
\end{align*}
are valid, for constants independent of $(z,w)\in\h_{m/n}\times \h_{m/n}$.
Consequently, the sub-Bergman kernel $K_j$ satisfies the estimate

\begin{align}\label{E:SubkernelEstimate}
\left|K_j(z,w)\right| \lesssim \frac{|t|^{2n-1-E_j+\frac{nj}{m}}}{|1-t|^2|t^n-s^m|^2}.
\end{align}
\medskip
From this, $L^p$ boundedness of each sub-Bergman projection $\ck_j: L^2(\h_{m/n}) \to \cs_j$ follows:

\begin{proposition}\label{P:boundednessSubprojections}
For all $p \in (\frac{2m+2n}{m-mE_j+2n+jn}, \frac{2m+2n}{m+mE_j-nj})$, $\ck_j$ is a bounded operator on $L^p(\h_{m/n})$.
\end{proposition}

\begin{proof}
This comes immediately from Proposition \ref{P:typeA_mapping} by taking $A = 2n -1-E_j+\frac{nj}{m}$.
\end{proof}

The range of $L^p$ boundedness for the full Bergman projection is obtained by taking the ``worst'' range associated to the sub-Bergman projections given by Proposition \ref{P:boundednessSubprojections}. To see this explicitly, recall
that $E_j=\left\lfloor \frac{(j+1)n-1}m\right\rfloor$, so
\begin{equation*}
\frac{n(j+1)-1}{m}-1 < E_j \leq \frac{n(j+1)-1}{m}\qquad\forall\, j\in\{0,\dots , m-1\}.  
\end{equation*}
As $m$ and $n$ are relatively prime, elementary number theory, \cite{HardyWrightbook}, Theorem 57, page 51, gives a unique $x\in\{0,\dots , m-1\}$ such that
$$n\, x\equiv 1\pmod m.$$
Note that $x\neq 0$ or $m-1$. Setting $j_0=x-1$, it follows that

\begin{equation}\label{E:specialE}
E_{j_0}=\frac{n(j_0+1)-1}m,
\end{equation}
and for all $j\neq j_0$ in $\{0,\dots , m-1\}$,

$$E_{j} <\frac{n(j+1)-1}m.$$
Thus $mE_{j_0}-nj_0 =n-1$ and Proposition \ref{P:boundednessSubprojections} says that $\ck_{j_0}$ is bounded on $L^p\left(\h_{m/n}\right)$ for $p \in (\frac{2m+2n}{m+n+1},\frac{2m+2n}{m+n-1})$.
It also says that the sub-Bergman projections $\ck_j$, $j\neq j_0$, are $L^p$ bounded for a larger H\" older symmetric interval about 2. Therefore, from \eqref{E:bergman_decomp} we obtain

\begin{corollary}\label{C:1halfMain}
The Bergman projection $\mathbf{B}_{m/n}$  is a bounded operator on $L^p(\h_{m/n})$ for all $p \in \left(\frac{2m+2n}{m+n+1},\frac{2m+2n}{m+n-1}\right)$.
\end{corollary}

The observations on $E_j$ and $E_{j_0}$, and \eqref{E:SubkernelEstimate}, also yield the following estimate on the full Bergman kernel

\begin{equation}\label{E:BergmankernelEstimate}
|\B_{m/n}(z,w)| \lesssim \frac{|t|^{2n-1+ \frac{1-n}{m}}}{|1-t|^2|t^n-s^m|^2}.
\end{equation}

\begin{remark}
In the next section, the range of $L^p$ boundedness in Corollary \ref{C:1halfMain} is sharp.  This implies estimate \eqref{E:BergmankernelEstimate} is optimal. 
\end{remark}

\begin{remark} We emphasize that the $L^p$ boundedness results in this section do not require cancellation properties of the kernels involved. Thus Proposition \ref{P:boundednessSubprojections} and Corollary \ref{C:1halfMain} also apply to the operators associated to $|K_j(z,w)|$ and $|\B_{m/n}(z,w)|$. It is interesting that, in {\it all} known cases where one implication can be proved, it holds that $\mathbf{B}_{\Omega}$ is bounded on $L^p(\Omega)$ if and only if $\left|\mathbf{B}_{\Omega}\right|$ is also bounded on $L^p(\Omega)$.  See \cite{Axl88}, \cite{McNSte94}, \cite{LanSte12} for further information. This equivalence is of course false for more general operators, e.g. the Szeg\"o projection or Cauchy-Leray integral.

\end{remark}


\section{The rational case: $L^p$ non-boundedness}\label{S:NonLpRational}
As in \cite{EdhMcN16}, we shall show that $\mathbf{B}_{m/n}$ fails to be $L^p$ bounded (for the range of $p$ indicated in Theorem \ref{T:MainLpRational}) by exhibiting a single function $f\in L^\infty\left(\h_{m/n}\right)$ such that
$\mathbf{B}_{m/n} f\notin L^p\left(\h_{m/n}\right)$.

The initial step is based on orthogonality and does not require $\gamma$ to be rational. Namely,
the rotational symmetry of $\h_\gamma$ implies that $\mathbf{B}_\gamma$ acts in a simple fashion on certain monomials in $z_1$ and $\bar z_2$:

\begin{proposition}\label{P:EasyMonomials}
If both $(\beta_1,\beta_2)$ and $(\beta_1, -\beta_2)$ belong to $\ca^2_\gamma$, then there exists a constant $C$ such that

\begin{equation*}
\mathbf{B}_\gamma\left(z_1^{\beta_1}\, \bar z_2^{\,\beta_2}\right) = C\, z_1^{\beta_1}\, z_2^{-\beta_2}.
\end{equation*}
\end{proposition}

\begin{proof} 
Let $f(z) = z_1^{\beta_1}\bar{z}_2^{\, \beta_2}$, $\ca=\ca^2_\gamma$, $\h=\h_\gamma$ and $H$ be the Reinhardt shadow of $\h_\gamma$ for short. A straightforward computation yields

\begin{align*}
\mathbf{B}_\gamma(f)(z) &=  \int_{\h} \sum_{\alpha \in \ca} \frac{z^{\alpha}\bar{w}^{\alpha}}{c_{\alpha}^2} f(w) \,dV(w) \\
	&= \sum_{\alpha \in \ca}\frac{z^{\alpha}}{c_{\alpha}^2} \int_{\h} w_1^{\beta_1}\bar{w}_1^{\alpha_1}\bar{w}_2^{\alpha_2+\beta_2} \,dV(w) \\
	&= \sum_{\alpha \in \ca} \frac{z^{\alpha}}{c_{\alpha}^2} \int_{\h} r_1^{\alpha_1+\beta_1+1} e^{i{\theta_1}{(\beta_1-\alpha_1)}}r_2^{\alpha_2+\beta_2+1} e^{-i{\theta_2}({\beta_2+\alpha_2})}\,dr \,d\theta \\
	&=\sum_{\alpha \in \ca} \frac{z^{\alpha}}{c_{\alpha}^2}\left( \int_0^{2\pi} e^{i{\theta_1}{(\beta_1-\alpha_1)}} \, d{\theta_1} \right) \left( \int_0^{2\pi} e^{-i{\theta_2}{(\beta_2+\alpha_2})} \, d{\theta_2} \right)\left( \int_{H} r_1^{\alpha_1+\beta_1+1}r_2^{\alpha_2+\beta_2+1} \, dr \right)\\
	&=C z_1^{\beta_1}z_2^{\,-\beta_2},
\end{align*}
where $C$ is a constant.  
\end{proof}

When $\gamma\in\Q^+$, a similar result on the subspaces $\cs_j$ holds, by the same proof:

\begin{proposition}\label{P:EasyMonomials2}
If both $(\beta_1,\beta_2)$ and $(\beta_1, -\beta_2)$ belong to $\cg_j$ for some $j\in\{0,1,\dots, m-1\}$, then there exists a constant $C$ such that

\begin{align*}
 \ck_l \left(z_1^{\beta_1}\, \bar z_2^{\,\beta_2}\right) = 
  \begin{cases}
 C\, z_1^{\beta_1}\, z_2^{-\beta_2}   , & l=j\\
    0, & l\neq j,
  \end{cases}
  \end{align*}
   for all $l\in\{0,1,\dots, m-1\}$.
\end{proposition}

Let $m,n \in \Z^+$ be relatively prime.  Multi-indices $(\alpha_1, \alpha_2)$ lying on the boundary line \eqref{E:RationalThreshold} of the lattice point diagram of $\h_{m/n}$, i.e. those indices satisfying 
$\alpha_2 = -\frac nm \alpha_1 - 1 + \frac{1-n}{m}$, are ``just barely'' in $\ca^2_{m/n}$. The case $\frac mn = \frac 32$ is illustrated below:

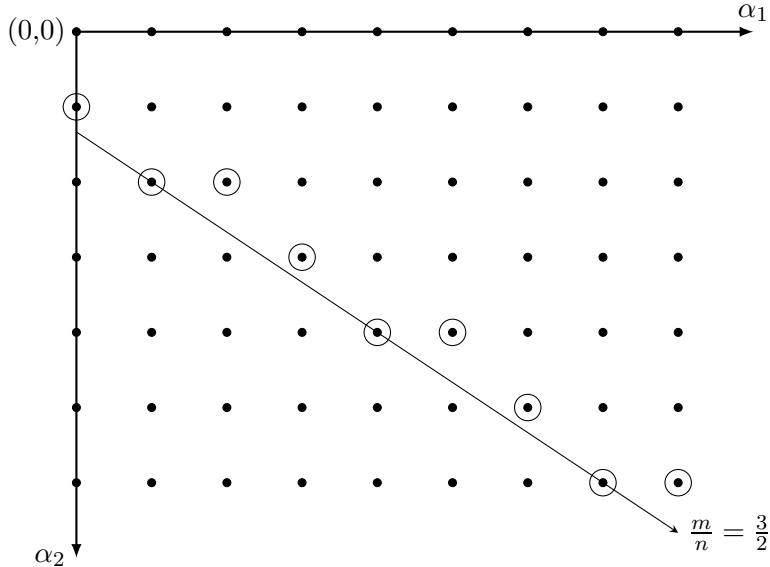
{\centering
\begin{tikzpicture}
\draw[-{latex}, thick] (0,0) -- (9,0) node[anchor=south] {$\alpha_1$};
\draw[-{latex}, thick] (0,0) -- (0,-7) node[anchor=east] {$\alpha_2$};

\draw[-stealth, thin] (0,-4/3) -- (8,-20/3) node[anchor=west] {$\frac mn = \frac 32$};

\filldraw[black] (0,0) circle (1.5pt) node[anchor=east] {(0,0)};
\filldraw[black] (1,0) circle (1.5pt) ;
\filldraw[black] (2,0) circle (1.5pt) ;
\filldraw[black] (3,0) circle (1.5pt) ;
\filldraw[black] (4,0) circle (1.5pt) ;
\filldraw[black] (5,0) circle (1.5pt) ;
\filldraw[black] (6,0) circle (1.5pt) ;
\filldraw[black] (7,0) circle (1.5pt) ;
\filldraw[black] (8,0) circle (1.5pt) ;

\filldraw[black] (0,-1) circle (1.5pt) ;
\filldraw[black] (1,-1) circle (1.5pt) ;
\filldraw[black] (2,-1) circle (1.5pt) ;
\filldraw[black] (3,-1) circle (1.5pt) ;
\filldraw[black] (4,-1) circle (1.5pt) ;
\filldraw[black] (5,-1) circle (1.5pt) ;
\filldraw[black] (6,-1) circle (1.5pt) ;
\filldraw[black] (7,-1) circle (1.5pt) ;
\filldraw[black] (8,-1) circle (1.5pt) ;

\filldraw[black] (0,-2) circle (1.5pt) ;
\filldraw[black] (1,-2) circle (1.5pt) ;
\filldraw[black] (2,-2) circle (1.5pt) ;
\filldraw[black] (3,-2) circle (1.5pt) ;
\filldraw[black] (4,-2) circle (1.5pt) ;
\filldraw[black] (5,-2) circle (1.5pt) ;
\filldraw[black] (6,-2) circle (1.5pt) ;
\filldraw[black] (7,-2) circle (1.5pt) ;
\filldraw[black] (8,-2) circle (1.5pt) ;

\filldraw[black] (0,-3) circle (1.5pt) ;
\filldraw[black] (1,-3) circle (1.5pt) ;
\filldraw[black] (2,-3) circle (1.5pt) ;
\filldraw[black] (3,-3) circle (1.5pt) ;
\filldraw[black] (4,-3) circle (1.5pt) ;
\filldraw[black] (5,-3) circle (1.5pt) ;
\filldraw[black] (6,-3) circle (1.5pt) ;
\filldraw[black] (7,-3) circle (1.5pt) ;
\filldraw[black] (8,-3) circle (1.5pt) ;

\filldraw[black] (0,-4) circle (1.5pt) ;
\filldraw[black] (1,-4) circle (1.5pt) ;
\filldraw[black] (2,-4) circle (1.5pt) ;
\filldraw[black] (3,-4) circle (1.5pt) ;
\filldraw[black] (4,-4) circle (1.5pt) ;
\filldraw[black] (5,-4) circle (1.5pt) ;
\filldraw[black] (6,-4) circle (1.5pt) ;
\filldraw[black] (7,-4) circle (1.5pt) ;
\filldraw[black] (8,-4) circle (1.5pt) ;

\filldraw[black] (0,-5) circle (1.5pt) ;
\filldraw[black] (1,-5) circle (1.5pt) ;
\filldraw[black] (2,-5) circle (1.5pt) ;
\filldraw[black] (3,-5) circle (1.5pt) ;
\filldraw[black] (4,-5) circle (1.5pt) ;
\filldraw[black] (5,-5) circle (1.5pt) ;
\filldraw[black] (6,-5) circle (1.5pt) ;
\filldraw[black] (7,-5) circle (1.5pt) ;
\filldraw[black] (8,-5) circle (1.5pt) ;

\filldraw[black] (0,-6) circle (1.5pt) ;
\filldraw[black] (1,-6) circle (1.5pt) ;
\filldraw[black] (2,-6) circle (1.5pt) ;
\filldraw[black] (3,-6) circle (1.5pt) ;
\filldraw[black] (4,-6) circle (1.5pt) ;
\filldraw[black] (5,-6) circle (1.5pt) ;
\filldraw[black] (6,-6) circle (1.5pt) ;
\filldraw[black] (7,-6) circle (1.5pt) ;
\filldraw[black] (8,-6) circle (1.5pt) ;

\draw[black] (0,-1) circle (5pt) ;
\draw[black] (1,-2) circle (5pt) ;
\draw[black] (2,-2) circle (5pt) ;
\draw[black] (3,-3) circle (5pt) ;
\draw[black] (4,-4) circle (5pt) ;
\draw[black] (5,-4) circle (5pt) ;
\draw[black] (6,-5) circle (5pt) ;
\draw[black] (7,-6) circle (5pt) ;
\draw[black] (8,-6) circle (5pt) ;

\end{tikzpicture}\\
}

There are three pieces of useful information that may be extracted from this lattice point diagram. First, consider the vertical line $\alpha_1=M$, $M\in\Z^+$, and the lattice point on it that is closest to
the boundary line $\alpha_2=-\frac nm\alpha_1-1+\frac{1-n}m$ (i.e., the points circled in the picture above). The $\alpha_2$ coordinate of this point was also defined in \eqref{E:ell(j)}.
The monomial corresponding to this point, $\left(M,\ell(M)\right)$, has the smallest range of
$L^p$ integrability, $p>2$, amongst all the monomials corresponding to lattice points on $\alpha_2=M$, $\alpha_2\geq -\frac nm\alpha_1-1+\frac{1-n}m$. Second, this range of $L^p$ integrability is the same
for all vertical lines $\alpha_1=M +km$, for $k\in\Z^+$. Finally, the closer the circled lattice point is to the boundary line, the smaller the range $p> 2$ for which $z^\alpha\in L^p\left(\h_{m/n}\right)$. When it actually lies on the boundary line,
the monomial $z^\alpha$ corresponding to this lattice point has the smallest range of $L^p$ integrability for all $\alpha\in\ca^2_{m/n}$.

The following results detail these observations:

\begin{proposition}\label{P:endpoints_are_worst}
Let $\left(M,\ell(M)\right)\in \ca^2_{m/n}$, where $\ell(M)$ is described above or, equivalently, defined by \eqref{E:ell(j)}.
Let $p>2$.

If $z_1^M\, z_2^{\ell(M)}\in L^p\left(\h_{m/n}\right)$, then $z_1^M\, z_2^{\alpha_2}\in L^p\left(\h_{m/n}\right)$ for all $\alpha_2\geq\ell(M)$.
\end{proposition}

\begin{proof}
Obvious, since $\left|z_2^{\alpha_2}\right|=\left|z_2^{\ell(M)}\right|\cdot|z_2|^{\alpha_2-\ell(M)}$ and $|z_2|^{\alpha_2-\ell(M)}\in L^\infty\left(\h_{m/n}\right)$.
\end{proof}

\begin{proposition}\label{P:congruents_are_same}
If $N\equiv M\pmod m$ and $z_1^M z_2^{\ell(M)}\in L^p\left(\h_{m/n}\right)$, then  $z_1^N z_2^{\ell(N)}\in L^p\left(\h_{m/n}\right)$.

Furthermore,

\begin{equation*}
\left\|z_1^N\, z_2^{\ell(N)}\right\|_p= \left\|z_1^M\, z_2^{\ell(M)}\right\|_p.
\end{equation*}
\end{proposition}

\begin{proof}
By \eqref{E:p_allowable}, the hypothesis implies

\begin{equation*}
\ell(M) > -\frac nm M-\frac 2p-\frac{2n}{mp}.
\end{equation*}
If $N=M+km$, for $k\in\Z^+$, the lattice point diagram shows that $\ell(N)=\ell(M)-kn$. Thus

\begin{align*}
kn+\ell(N) =\ell(M) &> -\frac nm M-\frac 2p -\frac{2n}{mp} \\
&= -\frac nm (N-km)-\frac 2p -\frac{2n}{mp}.
\end{align*}
This implies that $\left(N,\ell(N)\right)$ satisfies  \eqref{E:p_allowable}, so $z_1^N z_2^{\ell(N)}\in L^p\left(\h_{m/n}\right)$.

The equality of the $L^p$ norms follows by computing both expressions in polar coordinates.
\end{proof}

\begin{proposition}\label{P:NonLpKj} For each $j\in\{0, 1, \dots, m-1\}$, the sub-Bergman projection $\ck_j$ does not map $L^{\infty}(\h_{m/n})$ to $L^p(\h_{m/n})$
for any $p\geq \frac{2m+2n}{m+mE_j-nj}$.
\end{proposition}

\begin{proof}
Fix $j$, and take $\beta_1= j+km$ for some $k\in\Z^+\cup \{0\}$. Let $\beta_2=\ell(\beta_1)$, and note that \eqref{E:ell(j)} says that

\begin{equation*}
\beta_2 = -1 -nk -E_j <0.
\end{equation*}
Thus, $(\beta_1,\beta_2), (\beta_1, -\beta_2)\in\cg_j$. Let $f(z)=z_1^{\beta_1}\bar{z}_2^{-\beta_2}$; clearly $f\in L^\infty\left(\h_{m/n}\right)$. Proposition \ref{P:EasyMonomials2} says
that $\ck_j f= C z_1^{\beta_1} z_2^{\beta_2}$. 

Computing in polar coordinates

\begin{align*}
\int_{\h} \left|z_1^{\beta_1}z_2^{\beta_2}\right|^p\,dV(z) &= 4\pi^2 \int_{H} r_1^{p\beta_1+1}r_2^{p\beta_2+1}\,dV(z)\\
	&\approx \int_0^1 r_2^{p\beta_2+1} \int_0^{r_2^{n/m}} r_1^{p\beta_1+1}\,dr_1\,dr_2\\
	&\approx \int_0^1 r_2^{p\beta_2+1 + \frac{np\beta_1}{m} + \frac{2n}{m}} \, dr_2.
\end{align*}

This integral diverges when

\begin{equation}\label{E:conditionKj1}
p\beta_2+1 + \frac{np\beta_1}{m} + \frac{2n}{m} \le -1.
\end{equation}
Substituting $\beta_1= j+km$ and $\beta_2 = -1 -nk -E_j$, \eqref{E:conditionKj1} becomes

\begin{equation}\label{E:conditionKj2}
-p\left(m+mE_j-nj\right)\leq -2n-2m.
\end{equation}
However, since $E_j=\left\lfloor\frac{n(j+1)-1}n\right\rfloor$,

\begin{align*}
m+mE_j-nj &> m+m\left\{\frac{n(j+1)-1}m -1\right\} -nj \\ &=n-1 \geq 0,
\end{align*}
so \eqref{E:conditionKj2} is equivalent to $p\geq \frac{2m+2n}{m+mE_j-nj}$, which completes the proof.
\end{proof}

\begin{proposition}\label{L:NonLpBoundednessRational}
For $p \ge \frac{2m+2n}{m+n-1}$, $\mathbf{B}_{m/n}$ fails to map $L^{\infty}(\h_{m/n})$ to $L^p(\h_{m/n})$.
\end{proposition}

\begin{proof} 
Since $\gcd(m,n)=1$, the equation $nx\equiv 1\pmod m$ has a unique solution $x\in \{0,\dots m-1\}$. By the same modular arithmetic that led to \eqref{E:specialE},
there exists a unique multi-index $\left(j_0,\ell(j_0)\right)$ satisfying 

\begin{align}\label{E:specialBeta}
0 &\le j_0 \le m-1\notag \\ 
\ell(j_0) &= -\frac {n}{m}j_0 -1 + \frac{1-n}{m},
\end{align}
i.e., the lattice point $\left(j_0,\ell(j_0)\right)$ lies on the boundary line determining $\ca^2_{m/n}$. Proposition \ref{P:NonLpKj} says that $\ck_{j_0}$ does not map the bounded function
$g(z)=z_1^{j_0}\bar z_2^{-\ell(j_0)}$ to $L^p(\h_{m/n})$ for $p \ge \frac{2m+2n}{m+n-1}$. On the other hand, Proposition \ref{P:EasyMonomials2} says that $\ck_j(g)=0$ for all $j\neq j_0$. Thus,
\eqref{E:bergman_decomp} gives the claimed result.

\end{proof}

To obtain $L^p$ non-boundedness for $p<2$, recall an elementary consequence of the self-adjointness  of the Bergman projection (in the ordinary $L^2$ inner product):

\begin{lemma}\label{L:symmetricLp}  Let $\Omega$ be a domain and let $p > 1$.  If $\mathbf{B}$ maps $L^p(\Omega)$ to $A^p(\Omega)$ boundedly, then it also maps $L^q(\Omega)$ to $A^q(\Omega)$ boundedly, where $\frac 1p+ \frac 1q =1$.
\end{lemma}

\begin{proof}
Let $f \in L^q(\Omega)$.  Then

\begin{align*}
\left\|\mathbf{B}f\right\|_q &= \sup_{\norm{g}_p = 1} \left|\left\langle \mathbf{B}f, g \right\rangle\right| = \sup_{\norm{g}_p = 1} \left|\left\langle f, \mathbf{B}g \right\rangle\right| \\
	&\leq \sup_{\norm{g}_p = 1} \left( \norm{f}_q \norm{\mathbf{B}g}_p \right)
	\lesssim \norm{f}_q.
\end{align*}

\end{proof}

Proposition \ref{L:NonLpBoundednessRational} and Lemma \ref{L:symmetricLp} give the other half of Theorem \ref{T:MainLpRational}:

\begin{corollary}
$\mathbf{B}_{m/n}$ is not a bounded operator on $L^p(\h_{m/n})$ for  $p\notin(\frac{2m+2n}{m+n+1},\frac{2m+2n}{m+n-1})$.
\end{corollary}


\section{The irrational case: degenerate $L^p$ mapping}\label{S:Irrational}

The plausibility of Theorem \ref{T:MainLpIrrational} is already suggested by Theorem \ref{T:MainLpRational}. If $\gamma\notin\Q$, we may approximate $\gamma$ by rationals $\frac mn$ with $m+n$ tending to infinity  (keeping $\gcd(m,n)=1$).  However, Theorem \ref{T:MainLpRational} shows both that  the interval of $L^p$ boundedness of $\mathbf{B}_{m/n}$ depends on $m+n$ and that this interval shrinks to the point 2 as $m+n\to\infty$. 

To actually prove Theorem \ref{T:MainLpIrrational}, a more quantified version of this argument is necessary. For this, we use a classical theorem of Dirichlet on diophantine approximation. 
This result is proved, for instance, in \cite{HardyWrightbook} as Theorem 187 on page 158. 

\begin{proposition}[Dirichlet]\label{T:DirichletApprox}
If $\gamma$ is irrational, there exists a sequence of rational numbers $\left\{\frac{m_j}{n_j}\right\}$, with $\frac{m_j}{n_j} \to \gamma$, such that

\begin{equation*}
\left|\frac{n_j}{m_j}-\frac{1}{\gamma} \right|<\frac{1}{m_j^2}.
\end{equation*}
\end{proposition}
\medskip

\begin{proof}[Proof of Theorem \ref{T:MainLpIrrational}]
Fix $p > 2$.  We will exhibit an $f \in L^{\infty}(\h_{\gamma})$ such that $\mathbf{B}_{\gamma}(f) \notin L^p(\h_{\gamma}).$  

Let $\left\{\frac{m_j}{n_j}\right\}$ be a sequence of rational numbers given by Proposition \ref{T:DirichletApprox}. Temporarily fix the index $j$.  From \eqref{E:specialBeta}, there exists a unique $\beta = (\beta_1,\beta_2) \in \ca^2_{m_j/n_j}$ with $0 \le \beta_1 \le m_j-1$ and such that

\begin{equation}\label{E:BorderlineMultiIndex}
\beta_2 = \frac{1- n_j\beta_1 -n_j - m_j}{m_j} \in \Z.
\end{equation}
Assume for the moment that this multi-index $\beta \in \ca^2_{\gamma}$.  We will shortly show this is always the case. 

Let $f_j(z) := z_1^{\beta_1}/\bar{z}_2^{\beta_2}$; as $\beta_2 < 0$, $f_j \in L^{\infty}(\h_{\gamma})$. Since we are assuming $\beta \in \ca^2_{\gamma}$,  Proposition \ref{P:EasyMonomials} implies $\mathbf{B}_{\gamma}(f_j)(z) \approx z_1^{\beta_1}z_2^{\beta_2}$.  It follows that

\begin{align*}
\norm{\mathbf{B}_{\gamma}(f_j)}_{L^p(\h_{\gamma})}^p &\approx \int_{\h_{\gamma}}|z_1^{\beta_1p}z_2^{\beta_2p}|\,dV(z) = 4\pi^2\int_{H_{\gamma}}r_1^{\beta_1p+1}r_2^{\beta_2p+1}\,dr\\
	&\approx \int_0^1 r_2^{\beta_2p+1}\int_0^{r_2^{1/{\gamma}}}r_1^{\beta_1p+1}\,dr_1\,dr_2\\
	&\approx \int_0^1 r_2^{\beta_2p+1+\frac{\beta_1p}{\gamma}+\frac{2}{\gamma}}\,dr_2.
\end{align*}
This diverges if the exponent is $ \le -1$.  Substituting the expression for $\beta_2$ in \eqref{E:BorderlineMultiIndex} and rearranging terms, this happens exactly when

\begin{equation}\label{E:IrrationalExponentThreshold}
p\left(1+\frac{n_j-1}{m_j}+\beta_1\left(\frac {n_j}{m_j} -\frac{1}{\gamma}\right)\right) \ge 2+\frac{2}{\gamma}.
\end{equation}

Consider the left hand side of \eqref{E:IrrationalExponentThreshold}.  Since $0 \le \beta_1 \le m_j-1$, 
\begin{equation*}
\beta_1\left|\frac{n_j}{m_j} -\frac 1\gamma\right| < \frac 1{m_j},
\end{equation*}
by Proposition \ref{T:DirichletApprox}. Thus

\begin{align*}
p\left(1+\frac{n_j-1}{m_j}+\beta_1\left(\frac {n_j}{m_j} -\frac{1}{\gamma}\right)\right) &\ge p\left(1+\frac {n_j-1}{m_j} - \beta_1 \left| \frac {n_j}{m_j} - \frac{1}{\gamma}\right|\right)\\
	&> p\left(1+\frac{n_j-2}{m_j} \right).
\end{align*}
However since $p>2$, we can always choose $j$ large enough so that 

\begin{equation*}
p\left(1+\frac{n_j-2}{m_j} \right) > 2+\frac{2}{\gamma}.
\end{equation*}
Thus, \eqref{E:IrrationalExponentThreshold} is satisfied for such $j$, which shows $\mathbf{B}_{\gamma}(f_j) \notin L^p(\h_{\gamma})$.  

We now show that the unique multi-index $\beta = (\beta_1,\beta_2) \in \ca^2_{m_j/n_j}$ with $0 \le \beta_1 \le m_j-1$ and $\beta_2$ given by \eqref{E:BorderlineMultiIndex} is necessarily in $\ca^2_{\gamma}$.  We'll leave off the subscript $j$ in what follows.  

Again, the rational approximation $\left|\frac{n}{m}-\frac{1}{\gamma} \right|<\frac{1}{m^2}$ is essential.  If $\frac {m}{n} > \gamma$, then $A^2(\h_{m/n}) \subset A^2(\h_{\gamma})$ so automatically $\beta\in\ca^2_\gamma$.  Suppose instead that $\frac mn < \gamma$.  Lemma \ref{L:gammaAllowableMultiIndex} implies that $\beta \in \ca^2_{\gamma}$ if and only if $\beta_1 \ge 0$ and the lattice point corresponding to $\beta$ lies {\it strictly above} the line

\begin{align*}
g(\beta_1) := -\frac {\beta_1}{\gamma} - \frac{1}{\gamma} -1.
\end{align*}
But since $\frac mn \in \Q^+$, a multi-index $\beta \in \ca^2_{m/n}$ if and only if both $\beta_1 \ge 0$ and the lattice point corresponding to $\beta$ lies {\em on or above} the line
\begin{equation*}
h(\beta_1) := -\frac nm \beta_1 + \frac{1-n}{m} -1.
\end{equation*}
Now for $0 \le \beta_1 \le m-1$,
\begin{align*}
h(\beta_1) - g(\beta_1) &= \frac 1m - (\beta_1+1)\left(\frac nm - \frac{1}{\gamma} \right)\\ &\ge \frac 1m - m \left(\frac nm - \frac{1}{\gamma} \right)\\ & >0.
\end{align*}
From this it follows that $\beta = (\beta_1, \beta_2) \in \ca^2_{\gamma}$. 

Since $p>2$ was arbitrary, the above shows that $\mathbf{B}_\gamma$ is not $L^p$ bounded for any $p>2$.  Lemma \ref{L:symmetricLp} now shows $\mathbf{B}_\gamma$ is not $L^p$ bounded for any $1<p<2$, which completes the proof.
\end{proof}

\bigskip
\bigskip

\bibliographystyle{acm}
\bibliography{EdhMcN16}

\end{document}